\documentclass
[
    a4paper,
    DIV=11,
    abstracton,
    numbers=noendperiod
]
{scrartcl}

\usepackage
{
    graphicx,
    amssymb,
    amsmath,
    amsthm,
    dsfont, 
    mathrsfs,
    authblk,
    xcolor,
    enumitem,
    tikz-cd, 
    mathtools,
    ifthen,
}

\usepackage[pdffitwindow=false,
            plainpages=false,
            pdfpagelabels=true,
            pdfpagemode=UseOutlines,
            pdfpagelayout=SinglePage,
            bookmarks=false,
            colorlinks=true,
            hyperfootnotes=false,
            linkcolor=blue,
            urlcolor=blue!30!black,
            citecolor=green!50!black]{hyperref}

\usepackage[bf,normal]{caption}

\newtheorem{theorem}{Theorem}[section]
\newtheorem{corollary}[theorem]{Corollary}
\newtheorem{lemma}[theorem]{Lemma}
\newtheorem{proposition}[theorem]{Proposition}
\newtheorem{definition}[theorem]{Definition}
\theoremstyle{definition}
\newtheorem{example}[theorem]{Example}
\newtheorem{remark}[theorem]{Remark}

\newcommand{\tsp}{\hspace*{0.1em}} 

\newcommand{\subfiguretitle}[1]{{\scriptsize{#1}} \\[0.25ex]}
\newcommand{\R}{\mathbb{R}}                                     
\newcommand{\innerprod}[2]{\left\langle #1,\, #2 \right\rangle} 
\newcommand{\dd}{\mathrm{d}}                                    
\newcommand{\pp}[1]{\mathbb{#1}}                                
\providecommand{\norm}[1]{\left\lVert #1 \right\rVert}          

\newcommand{\exsum}{\oplus}                                     

\newcommand\restr[2]{{\left.\kern-\nulldelimiterspace #1 \vphantom{\big|} \right|_{#2}}} 

\newcommand\xqed[1]{\leavevmode\unskip\penalty9999 \hbox{}\nobreak\hfill \quad\hbox{#1}}
\newcommand{\exampleSymbol}{\xqed{$\blacktriangle$}} 

\newcommand{\inspace}{\mathbb{X}} 
\newcommand{\outspace}{\mathbb{Y}} 

\newcommand{\id}{\mathcal{I}} 
\newcommand{\idop}{\mathcal{I}} 

\newcommand{\HS}{\text{HS}}         

\newcommand{\ebd}[1][]{
   \ifthenelse{\equal{#1}{}}{\mathcal{E}}{\mathcal{E}_{#1}}}
\newcommand{\pro}[1][]{
   \ifthenelse{\equal{#1}{}}{\mathcal{Q}}{\mathcal{Q}_{#1}}}

\newcommand{\cme}{\mathcal{U}_{\scriptscriptstyle Y \mid X}} 
\newcommand{\ecme}{\widehat{\mathcal{U}}_{\scriptscriptstyle Y \mid X}} 

\newcommand{\pf}[1][]{
   \ifthenelse{\equal{#1}{}}{\mathcal{P}}{\mathcal{P}_{#1}}}
\newcommand{\epf}[1][]{
   \ifthenelse{\equal{#1}{}}{\widehat{\mathcal{P}}}{\widehat{\mathcal{P}}_{#1}}}
\newcommand{\ko}[1][]{
   \ifthenelse{\equal{#1}{}}{\mathcal{K}}{\mathcal{K}_{#1}}}
\newcommand{\eko}[1][]{
   \ifthenelse{\equal{#1}{}}{\widehat{\mathcal{K}}}{\widehat{\mathcal{K}}_{#1}}}

\newcommand{\cov}[1][]{\mathcal{C}_\mathit{\scriptscriptstyle #1}} 
\newcommand{\ecov}[1][]{\widehat{\mathcal{C}}_\mathit{\scriptscriptstyle #1}} 

\newcommand{\gram}[1][]{G_\mathit{\scriptscriptstyle #1}} 

\DeclareMathOperator{\mspan}{span}
\DeclareMathOperator{\diag}{diag}
\DeclareMathOperator{\rank}{rank}
\DeclareMathOperator{\mker}{ker}
\DeclareMathOperator{\mdim}{dim}

\DeclareMathOperator*{\argmin}{arg\,min} 

\mathtoolsset{centercolon}

\allowdisplaybreaks

\title{Singular Value Decomposition of Operators on Reproducing Kernel Hilbert Spaces}

\author[1]{Mattes Mollenhauer}
\author[2]{Ingmar Schuster}
\author[1]{\\Stefan Klus}
\author[1,3]{Christof Sch\"utte}
\affil[1]{\normalsize Department of Mathematics and Computer Science, Freie Universit\"at Berlin, Germany}
\affil[2]{Zalando Research, Zalando SE Berlin, Germany}
\affil[3]{\normalsize Zuse Institute Berlin, Germany}
\date{}

\begin{document}

\maketitle

\begin{abstract}
Reproducing kernel Hilbert spaces (RKHSs) play an important role in many statistics and machine learning applications ranging from support vector machines to Gaussian processes and kernel embeddings of distributions. Operators acting on such spaces are, for instance, required to embed conditional probability distributions in order to implement the kernel Bayes rule and build sequential data models. It was recently shown that transfer operators such as the Perron--Frobenius or Koopman operator can also be approximated in a similar fashion using covariance and cross-covariance operators and that eigenfunctions of these operators can be obtained by solving associated matrix eigenvalue problems. The goal of this paper is to provide a solid functional analytic foundation for the eigenvalue decomposition of RKHS operators and to extend the approach to the singular value decomposition. The results are illustrated with simple guiding examples.
\end{abstract}

\section{Introduction}
\label{sec:Introduction}

A majority of the characterizing properties of a linear map such as range, null space, numerical condition, and different operator norms can be obtained by computing the \emph{singular value decomposition} (SVD) of the associated matrix representation. Furthermore, the SVD is used to optimally approximate matrices under rank constraints, solve least squares problems, or to directly compute the \emph{Moore--Penrose pseudoinverse}. Applications range from solving systems of linear equations and optimization problems and to a wide variety of methods in statistics, machine learning, signal processing, image processing, and other computational disciplines.

Although the matrix SVD can be extended in a natural way to compact operators on Hilbert spaces~\cite{Reed80}, this infinite-dimensional generalization is not as multifaceted as the finite-dimensional case in terms of numerical applications. This is mainly due to the complicated numerical representation of infinite-dimensional operators and the resulting problems concerning the computation of their SVD. As a remedy, one usually considers finite-rank operators based on finite-dimensional subspaces given by a set of fixed basis elements. The SVD of such finite-rank operators will be the main focus of this paper. We will combine the theory of the SVD of finite-rank operators with the concept of reproducing kernel Hilbert spaces, a special class of Hilbert spaces allowing for a high-dimensional representation of the abstract mathematical notion of ``data'' in a feature space. A significant part of the theory of RKHSs was originally developed in a functional analytic setting \cite{aronszajn50reproducing} and made its way into pattern recognition and statistics \cite{Schoe01,Berlinet04:RKHS,StCh08}. RKHSs are often used to derive nonlinear extensions of linear methods by embedding observations into a high-dimensional feature space and rewriting the method in terms of the inner product of the RKHS. This strategy is known as the \emph{kernel trick}~\cite{Schoe01}. The approach of embedding a countable number of observations can be generalized to the embedding of probability distributions associated with random variables into the RKHS \cite{Smola07Hilbert}. The theory of the resulting \emph{kernel mean embedding} (see \cite{MFSS17} for a comprehensive review), \emph{conditional mean embedding}~\cite{SHSF09,Gruen12,Klebanov2019rigorous,Park2020MeasureTheoretic} and \emph{Kernel Bayes rule}~\cite{Fukumizu13:KBR,Fukumizu15:NBI} spawned a wide range of nonparametric approaches to problems in statistics and machine learning. Recent advances based on the conditional mean embedding show that data-driven methods in various fields such as transfer operator theory, time series analysis, and image and text processing naturally give rise to a spectral analysis of finite-rank RKHS operators \cite{KSM19, KHM19}.

Practical applications of these spectral analysis techniques include the identification of the slowest relaxation processes of dynamical systems, e.g., conformational changes of complex molecules or slowly evolving coherent patterns in fluid flows, but also dimensionality reduction and blind source separation. The eigendecomposition, however, is beneficial only in the case where the underlying system is ergodic with respect to some density. If this is not the case, however, i.e., the stochastic process is time-inhomogeneous, eigendecompositions can be replaced by singular value decompositions in order to obtain similar information about the global dynamics~\cite{KoWuNoSch18}. Moreover, outside of the context of stochastic processes, the conditional mean embedding operator has been shown to be the solution of certain vector-valued regression problems~\cite{Gruen12,Park2020MeasureTheoretic}. Contrary to the transfer operator setting, input and output space can differ fundamentally (e.g., the input space could be text) and the constraint that the RKHS for input and output space must be identical is too restrictive. The SVD of RKHS operators does not require this assumption and is hence a more general analysis tool applicable to operators that solve regression problems and to transfer operators associated with more general stochastic processes.

In this paper, we will combine the functional analytic background of the Hilbert space operator SVD and
the theory of RKHSs to develop a self-contained and rigorous mathematical framework for the SVD of finite-rank operators acting on RKHSs and show that the SVD of such operators can be computed numerically by solving an auxiliary matrix eigenvalue problem. The remainder of the paper is structured as follows: Section~\ref{sec:Preliminaries} briefly recapitulates the theory of compact operators. In Section~\ref{sec:Decompositions of RKHS operators}, RKHS operators and their eigendecompositions and singular value decompositions will be described. Potential applications are discussed in Section~\ref{sec:Applications}, followed by a brief conclusion and a delineation of open problems in Section~\ref{sec:Conclusion}.

\section{Preliminaries}
\label{sec:Preliminaries}

In what follows, let $ H $ be a real Hilbert space, $ \innerprod{\cdot}{\cdot}_H $ its inner product, and $ \norm{\cdot}_H $ the induced norm. For a Hilbert space $ H $, we call a set $ \{ h_i \}_{i \in I} \subseteq H $ with an index set $ I $ an \emph{orthonormal system} if $ \innerprod{h_i}{h_j}_H = \delta_{ij} $ for all $ i,j \in I $. If additionally $ \mspan\{ h_i \}_{i \in I} $ is dense in $ H $, then we call $ \{ h_i \}_{i \in I} $ a \emph{complete orthonormal system}. If $ H $ is separable, then the index set $ I $ of every complete orthonormal system of $ H $ is countable. Given a complete orthonormal system, every $ x \in H $ can be expressed by the series expansion $ x = \sum_{i \in I} \innerprod{h_i}{x}_H h_i $.

\begin{definition} \label{defi:tensor_product}
Given two Hilbert spaces $ H $ and $ F $ and nonzero elements $ x \in H $ and $ y \in F $, we define the \emph{tensor product operator} $ y \otimes x \colon H \to F $ by $ (y \otimes x) \tsp h = \innerprod{x}{h}_H \tsp y $.
\end{definition}

Note that tensor product operators are bounded linear operators. Boundedness follows from the Cauchy--Schwarz inequality on $ H $. We define $ \mathcal{E} := \mspan \{ y \otimes x \mid x \in H, \, y \in F\} $ and call the completion of $ \mathcal{E} $ with respect to the inner product
\begin{equation*} \label{eq:hs_innerproduct}
    \innerprod{y_1 \otimes x_1}{y_2 \otimes x_2}:= \innerprod{y_1}{y_2}_F \innerprod{x_1}{x_2}_H
\end{equation*}
the \emph{tensor product} of the spaces $ F $ and $ H $, denoted by $ F \otimes H $. It follows that $ F \otimes H $ is again a Hilbert space. It is well known that, given a self-adjoint compact operator $ A \colon H \to H $, there exists an \emph{eigendecomposition} of the form
\begin{equation*}
    A = \sum\limits_{i \in I} \lambda_i (e_i \otimes e_i),
\end{equation*}
where $ I $ is either a finite or countably infinite ordered index set, $ \{e_i\}_{i \in I} \subseteq H $ an orthonormal system, and $ \{\lambda_i\}_{i \in I} \subseteq \R \setminus\{0\} $ the set of nonzero eigenvalues. If the index set $ I $ is not finite, then the resulting sequence $ (\lambda_i)_{i \in I} $ is a null sequence. Similarly, given a compact bounded operator $ A \colon H \to F $, there exists a \emph{singular value decomposition} given by
\begin{equation*}
    A = \sum\limits_{i \in I} \sigma_i (u_i \otimes v_i),
\end{equation*}
where $ I $ is again an either finite or countably infinite ordered index set, $ \{v_i\}_{i \in I} \subseteq H $ and $ \{u_i\}_{i \in I} \subseteq F $ two orthonormal systems, and $ \{\sigma_i\}_{i \in I} \subseteq \R_{>0} $ the set of singular values. As for the eigendecomposition, the sequence $ (\sigma_i)_{i \in I} $ is a null sequence if $ I $ is not finite. Without loss of generality, we assume the singular values of compact operators to be ordered in non-increasing order, i.e., $ \sigma_i \geq \sigma_{i+1} $. We additionally write $ \sigma_i(A) $ for the $ i $th singular value of a compact operator $ A $ if we want to emphasize to which operator we refer.
The following result shows the connection of the eigendecomposition and the SVD of compact operators.

\begin{lemma} \label{lem:eigendecomposition_svd}
Let $ A \colon H \to F $ be compact and let $ \{\lambda_i\}_{i \in I} $ denote the set of nonzero eigenvalues of $ A^*A $ counted with their multiplicities and $ \{v_i\}_{i \in I} $ the corresponding normalized eigenfunctions of $ A^*A $, then, for $ u_i := \lambda_i^{-1/2} A \tsp v_i $, the singular value decomposition of $ A $ is given by
\begin{equation*}
    A = \sum_{i \in I} \lambda_i^{1/2} \tsp (u_i \otimes v_i).
\end{equation*}
\end{lemma}

A bounded operator $ A \colon H \to F $ is said to be $ r $-dimensional if $ \rank(A) = r $. If $ r < \infty $, we say that $ A $ is \emph{finite-rank}.

\begin{theorem}[see \cite{Weidmann}] \label{theo:finite_rank_operators}
Let $ H $ and $ F $ be two Hilbert spaces and $ A \colon H \to F $ a linear operator. The operator $ A $ is finite-rank with $ \rank(A) = r $ if and only if there exist linearly independent sets $ \{ h_i \}_{1 \leq i \leq r} \subseteq H $ and $ \{ f_i \}_{1 \leq i \leq r} \subseteq F $ such that $ A = \sum_{i=1}^r {f_i \otimes h_i} $. Furthermore, then $ A^* = \sum_{i=1}^r h_i \otimes f_i $.
\end{theorem}

The class of finite-rank operators is a dense subset of the class of compact operators with respect to the operator norm.

\begin{definition} \label{defi:HS}
Let $ H $ and $ F $ be Hilbert spaces and $ \{h_i\}_{i \in I} \subseteq H $ be a complete orthonormal system. An operator $ A \colon H \to F $ is called a \emph{Hilbert--Schmidt operator} if $ \sum_{i \in I} \norm{A \tsp h_i}^2_F < \infty $.
\end{definition}

The space of Hilbert--Schmidt operators from $H$ to $F$ is itself a Hilbert space with the inner product $\innerprod{A}{B}_\HS:=  \sum_{i \in I} \innerprod{A \tsp h_i}{B \tsp h_i}_F$. Furthermore, it is isomorphic to the tensor product space $ F \otimes H $. The space of finite-rank operators is a dense subset of the Hilbert--Schmidt operators with respect to the Hilbert--Schmidt norm. Furthermore, every Hilbert--Schmidt operator is compact and therefore admits an SVD.

\begin{remark}
Based on the definitions of the operator norm and the Hilbert--Schmidt norm, we have $ \norm{A} = \sigma_1(A) $ for any compact operator and $ \norm{A}_\HS = \big(\sum_{i \in I} \sigma_i(A)^2 \big)^{1/2}$ for any Hilbert--Schmidt operator.
\end{remark}

We will now derive an alternative characterization of the SVD of compact operators by generalizing a classical block-matrix decomposition approach to compact operators. For the matrix version of this result, we refer the reader to~\cite{Golub13}. For two Hilbert spaces $H$ and $F$, we define the \emph{external direct sum} $ F \exsum H $ as the Hilbert space of tuples of the form $ (f, h) $, where $h \in H $ and $ f \in F$, with the inner product
\begin{equation*}
    \innerprod{(f,h)}{(f',h')}_{\exsum}:=\innerprod{h}{h'}_H + \innerprod{f}{f'}_F.
\end{equation*}
If $A \colon H \to F$ is a compact operator, then the operator $ T \colon F \exsum H \rightarrow F \exsum H $, with
\begin{equation} \label{eq:blockoperator}
  (f, h) \mapsto (A h,A^*f)
\end{equation}
is compact and self-adjoint
with respect to $\innerprod{\cdot}{\cdot}_{\exsum}$.
By interpreting the elements of $F \exsum H$ as column vectors
and generalizing algebraic matrix operations, we may rewrite the action
of the operator $ T $ on $ (f, h) $ in a block operator notation as
\begin{equation*}
  \begin{bmatrix}
    & A \\
    A^* &
  \end{bmatrix}
  \begin{bmatrix}
    f \\ h
  \end{bmatrix} =
  \begin{bmatrix}
    A h \\ A^* f
  \end{bmatrix}.
\end{equation*}
We remark that the block operator notation should be applied with caution since
vector space operations amongst $h \in H$ and $f \in F$ in terms of the matrix multiplication are only defined
if $F \exsum H$ is an internal direct sum.

\begin{lemma} \label{lem:block_operator_svd}
Let $A \colon H \rightarrow F$ be a compact operator and $T \colon F \exsum H \rightarrow F \exsum H$ be the block-operator given by \eqref{eq:blockoperator}. If $A$ admits the
SVD
\begin{align} \label{eq:block_svd}
    A = \sum\limits_{i \in I} \sigma_i (u_i \otimes v_i)
\end{align}
then $T$ admits the eigendecomposition
\begin{align} \label{eq:block_eigendecomposition}
    T = \sum\limits_{i \in I}
        \sigma_i \left[ \tfrac{1}{\sqrt{2}} (u_i,v_i) \otimes  \tfrac{1}{\sqrt{2}}(u_i,v_i) \right]
        - \sigma_i \left[ \tfrac{1}{\sqrt{2}}(-u_i,v_i) \otimes  \tfrac{1}{\sqrt{2}}(-u_i,v_i) \right].
\end{align}
\end{lemma}

A proof of this lemma can be found in Appendix~\ref{ap:block_operator_svd}.

\begin{corollary} \label{cor:block_operator_svd}
Let $A \colon H \rightarrow F$ be a compact operator. If $\sigma > 0$ is an eigenvalue of the block-operator $T \colon F \exsum H \rightarrow F \exsum H$ given by \eqref{eq:blockoperator} with the corresponding eigenvector $(u,v) \in F \otimes H$, then $\sigma$ is a singular value of $A$ with the corresponding left and right singular vectors $\norm{u}^{-1}_F u \in F$ and $\norm{v}_H^{-1}v \in H$.
\end{corollary}

\section{Decompositions of RKHS operators}
\label{sec:Decompositions of RKHS operators}

We will first introduce reproducing kernel Hilbert spaces, and then consider empirical operators defined on such spaces. The main results of this section are a basis orthonormalization technique via a kernelized QR decomposition in Section~\ref{sec:Kernel QR} and the eigendecomposition and singular value decomposition of empirical RKHS operators in Section~\ref{sec:eigendecomposition} and Section~\ref{sec:SVD} via auxiliary problems, respectively. The notation is adopted from~\cite{MFSS17, KSM19} and summarized in Table~\ref{tab:Notation}.

\subsection{RKHS}

The following definitions are based on \cite{StCh08, Schoe01}. In order to distinguish reproducing kernel Hilbert spaces from standard Hilbert spaces, we will use script style letters for the latter, i.e., $ \mathscr{H} $ and $ \mathscr{F} $.

\begin{table}[htb]
    \centering
    \caption{Overview of notation.}
    \begin{tabular}{l@{\hspace{3em}}c@{\hspace{3em}}c}
        \hline
        random variable & $ X $                     & $ Y $               \\
        domain          & $ \inspace $              & $ \outspace $       \\
        observation     & $ x $                     & $ y $               \\
        kernel function & $ k(x, x^\prime) $        & $ l(y, y^\prime) $  \\
        feature map     & $ \phi(x) $               & $ \psi(y) $         \\
        feature matrix  & $ \Phi = [\phi(x_1), \dots, \phi(x_m)] $ & $ \Psi = [\psi(y_1), \dots, \psi(y_n)] $ \\
        Gram matrix     & $ \gram[\Phi] = \Phi^\top \Phi $         & $ \gram[\Psi] = \Psi^\top \Psi $           \\
        RKHS            & $ \mathscr{H} $            & $ \mathscr{F} $      \\
        \hline
    \end{tabular}
    \label{tab:Notation}
\end{table}

\begin{definition}[Reproducing kernel Hilbert space, \cite{Schoe01}] \label{def:RKHS}
Let $ \inspace $ be a set and $ \mathscr{H} $ a space of functions $ f \colon \inspace \to \R $. Then $ \mathscr{H} $ is called a \emph{reproducing kernel Hilbert space (RKHS)} with corresponding inner product $ \innerprod{\cdot}{\cdot}_\mathscr{H} $ if a function $ k \colon \inspace \times \inspace \to \R $ exists such that
\begin{enumerate}[label=(\roman*), itemsep=0ex, topsep=1ex]
\item $ \innerprod{f}{k(x, \cdot)}_\mathscr{H} = f(x) $ for all $ f \in \mathscr{H} $ and
\item $ \mathscr{H} = \overline{\mspan\{k(x, \cdot) \mid x \in \inspace \}} $.
\end{enumerate}
\end{definition}

The function $ k $ is called \emph{reproducing kernel} and the first property the \emph{reproducing property}. It follows in particular that $ k(x, x^\prime) = \innerprod{k(x, \cdot)}{k(x^\prime,\cdot)}_\mathscr{H} $. The canonical feature map $ \phi \colon \inspace \to \mathscr{H} $ is given by $ \phi(x) := k(x, \cdot) $. Thus, we obtain $ k(x, x^\prime) = \innerprod{\phi(x)}{\phi(x^\prime)}_\mathscr{H} $. It was shown that an RKHS has a unique symmetric and positive definite kernel with the reproducing property and, conversely, that a symmetric positive definite kernel $ k $ induces a unique RKHS with $ k $ as its reproducing kernel~\cite{aronszajn50reproducing}. We will refer to the set $ \inspace $ as the corresponding \emph{observation space}.

\subsection{RKHS operators}

Finite-rank operators can be defined by a finite number of fixed basis elements in the corresponding RKHSs. In practice, finite-rank RKHS operators are usually estimates of infinite-dimensional operators based on a set of empirical observations. We later refer to this special type of finite-rank operator as \emph{empirical RKHS operator} although the concepts in this section are more general and do not need the assumption of the data in the observation space being given by random events.

Let $ \mathscr{H} $ and $ \mathscr{F} $ denote RKHSs based on the observation spaces $ \inspace $ and $ \outspace $, respectively. Given $ x_1, \dots, x_m  \in \inspace $ and $ y_1, \dots, y_n \in \outspace $, we call
\begin{equation*}
    \Phi := [\phi(x_1), \dots, \phi(x_m)] \quad \text{and} \quad \Psi := [\psi(y_1), \dots, \psi(y_n)]
\end{equation*}
their associated \emph{feature matrices}. Note that feature matrices are technically not matrices but row vectors in $ \mathscr{H}^m $ and $ \mathscr{F}^n $, respectively. Since the embedded observations in the form of $ \phi(x_i) \in \mathscr{H} $ and $ \psi(y_j) \in \mathscr{F} $ can themselves be interpreted as (possibly infinite-dimensional) vectors, the term \emph{feature matrix} is used. In what follows, we assume that feature matrices contain linearly independent elements. This is, for example, the case if $ k(\cdot, \cdot) $ is a radial basis kernel and the observations $ x_1, \dots, x_m \in \inspace $ consist of pairwise distinct elements. Given the feature matrices $ \Phi $ and $ \Psi $, we can define the corresponding Gram matrices by $ \gram[\Phi] = \Phi^\top \Phi \in \R^{m \times m} $ and $ \gram[\Psi] = \Psi^\top \Psi \in \R^{n \times n} $. That is, $ [\gram[\Phi]]_{ij} = k(x_i, x_j) $ and $ [\gram[\Psi]]_{ij} = l(y_i, y_j) $. We will now analyze operators $ S \colon \mathscr{H} \to \mathscr{F} $ of the form $ S = \Psi B \Phi^\top $, where $ B \in \R^{n \times m} $. Given $ v \in \mathscr{H} $, we obtain
\begin{equation*}
    S v = \Psi B \Phi^\top v
        = \sum_{i = 1}^n \psi(y_i) \sum_{j=1}^m b_{ij} \innerprod{\phi(x_j)}{v}_\mathscr{H}.
\end{equation*}
We will refer to operators $ S $ of this form as \emph{empirical RKHS operators}. Examples of such operators are described in Section~\ref{sec:Applications}.

\begin{remark} \label{rem:independent_features}
If the rows of $ B $ are linearly independent in $ \R^m $, then the elements of $ B \tsp \Phi^\top $ are linearly independent in $ \mathscr{H} $. The analogue statement holds for linearly independent columns of $ B $ and elements of $ \Psi \tsp B $ in $ \mathscr{F} $.
\end{remark}

\begin{proposition} \label{prop:rkhs_operator_properties}
The operator $ S $ defined above has the following properties:
\begin{enumerate}[label=(\roman*)]
\item $ S $ is a finite-rank operator. In particular, $ \rank(S) = \rank(B) $.
\item $ S^* = \Phi B^\top \Psi^\top $.
\item Let $ B = W \Sigma Z^\top $ be the the singular value decomposition of $ B $, where $ W = [\mathbf{w}_1, \dots, \mathbf{w}_n] $, $ \Sigma = \diag(\sigma_1, \dots, \sigma_r, 0, \dots, 0) $, and $ Z = [\mathbf{z}_1, \dots, \mathbf{z}_m] $, then
\begin{equation*}
    \norm{S} \leq \sum\limits_{i=1}^r \sigma_i \norm{\Psi \mathbf{w}_i}_\mathscr{F} \norm{\Phi \mathbf{z}_i}_\mathscr{H}.
\end{equation*}
\end{enumerate}
\end{proposition}

\begin{proof}
The linearity of $ S $ follows directly from the linearity of the inner product in $ \mathscr{H} $. We now show that properties (i)--(iii) can directly be obtained from Theorem~\ref{theo:finite_rank_operators}. Using $ B = W \Sigma Z^\top $, we can write $ S = (\Psi W) \Sigma (Z^\top \Phi^\top)$ and obtain
\begin{equation} \label{eq:finite_rank_empirical_form}
    S v = \sum\limits_{i=1}^r \sigma_i \Psi \mathbf{w}_i \innerprod{\Phi \mathbf{z}_i}{v}_\mathscr{H} \text{ for all } v \in \mathscr{H}.
\end{equation}
Since the elements in $ \Phi $ and $ \Psi $ are linearly independent, we see that $ \Phi Z $ and $ \Psi W $ are also feature matrices containing the linearly independent elements $ \Phi \tsp \mathbf{z}_i \in \mathscr{H} $ and $ \Psi \tsp \mathbf{w}_i \in \mathscr{F} $ as stated in Remark~\ref{rem:independent_features}. Therefore, \eqref{eq:finite_rank_empirical_form} satisfies the assumptions in Theorem \ref{theo:finite_rank_operators} if we choose $\{ \Phi \tsp \mathbf{z}_i \}_{1 \leq i \leq r} \subseteq \mathscr{H} $ and $\{ \sigma_i \Psi \tsp \mathbf{w}_i \}_{1 \leq i \leq r} \subseteq \mathscr{F} $ to be the required linearly independent sets.
Theorem \ref{theo:finite_rank_operators} directly yields all the desired statements.
\end{proof}

Note that the characterization \eqref{eq:finite_rank_empirical_form} is in general not a singular value decomposition of $ S $ since the given basis elements in $ \Phi Z $ and $ \Psi W $ are not necessarily orthonormal systems in $ \mathscr{H} $ and $ \mathscr{F} $, respectively.

\subsection{Basis orthonormalization and kernel QR decomposition}
\label{sec:Kernel QR}
When we try to perform any type of decomposition of the operator $S = \Psi B \Phi^\top$, we face the problem that the representation matrix $B$ is defined to work on feature matrix entries of $ \Psi$ and $\Phi$, which are not necessarily orthonormal systems
in the corresponding RKHSs. This leads to the fact that we can not simply 
decompose $ B $ with standard numerical routines based on the Euclidean inner products and expect a meaningful equivalent decomposition of $S$ in terms of RKHS inner products.
We therefore orthonormalize the feature matrices with respect to to the RKHS inner products and capture these transformations in a new representation matrix $\tilde{B}$ which allows using
matrix decompositions to obtain operator decompositions of~$S$.
We now generalize the matrix QR decomposition to feature matrices,
which is essentially equivalent to a kernelized Gram--Schmidt procedure~\cite{STJ04}. By expressing empirical RKHS operators with respect
to orthonormal feature matrices, we can perform operator decompositions in terms of a simple matrix decomposition.

\begin{proposition}[Kernel QR decomposition] \label{prop:kernelQR}
    Let $\Phi \in \mathscr{H}^m$ be a feature matrix. Then there exists a unique upper triangular matrix $ R \in \R^{m \times m} $ with strictly positive diagonal elements
    and a feature matrix $ \tilde{\Phi} \in \mathscr{H}^m$, such that
    \begin{equation*}
    \Phi = \tilde{\Phi} R
    \end{equation*} and $\tilde{\Phi}^\top  \tilde{\Phi}= I_m$.
\end{proposition}

\begin{proof}
    We have assumed that elements in feature matrices are linearly independent. Therefore $\Phi^\top \Phi$ is strictly positive definite. We have 
    a Cholesky decomposition  $\Phi^\top \Phi = R^\top R$ for a unique upper triangular matrix $ \R^{m \times m} $ with positive diagonal entries.
    By setting $ \tilde{\Phi} := \Phi R^{-1}$ and observing that $ \tilde{\Phi}^\top  \tilde{\Phi} = (\Phi R^{-1})^\top \Phi R^{-1} = I_m$,
    the claim follows. 
\end{proof}

By using Proposition~\ref{prop:kernelQR}, we can express empirical operators in orthonormalized basis elements. Given an empirical RKHS operator
$S = \Psi B \Phi^\top$ and the two corresponding kernel QR decompositions $\Phi = \tilde{\Phi}R_\Phi$ and  $\Psi = \tilde{\Psi}R_\Psi$,
we can rewrite
\begin{equation} \label{eq:QR_operator}
S = (\tilde{\Psi} R_\Psi^{-1}) B (\tilde{\Phi} R_\Phi^{-1})^\top = \tilde{\Psi} (R_\Psi^{-1} B (R_\Phi^{-1})^\top) \tilde{\Phi}^\top =  
\tilde{\Psi} \tilde{B} \tilde{\Phi}^\top.
\end{equation}

We can now simply perform any type of matrix decomposition on the new representation matrix $ \tilde{B}:= R_\Psi^{-1} B (R_\Phi^{-1})^\top$ to obtain an equivalent decomposition of the operator $S$. As examples, we give the SVD and the eigendecomposition
of $S$.

\begin{corollary}[Singular value decomposition]
    Let $ S = \tilde{\Psi} \tilde{B} \tilde{\Phi}^\top \colon \mathscr{H} \rightarrow \mathscr{F}$ be given by orthonormalized basis elements as above. 
    If $ \tilde{B} = \sum_{i=1}^r \sigma_i u_iv_i^\top $ is the singular value decomposition of $ \tilde{B} $,
    then 
    \begin{equation*}
    S = \sum\limits_{i=1}^r \sigma_i (\tilde{\Psi}u_i \otimes \tilde{\Phi}v_i)
    \end{equation*}
    is the singular value decomposition of $S$.
\end{corollary}

For the eigendecomposition, we require the operator to be a mapping from $ \mathscr{H} $ to itself. 
We will assume that both the domain and the range of $ S $ are defined via the same feature matrix $ \Phi $.
We consider the self-adjoint case, that is $ B $ (or equivalently~$\tilde{B}$) is symmetric.

\begin{corollary}[Eigendecomposition] \label{cor:eigendecomposition}
    Let $ S = \tilde{\Phi} \tilde{B} \tilde{\Phi}^\top \colon \mathscr{H} \rightarrow \mathscr{H}$ be given by orthonormalized basis elements as above. 
    Let $\tilde{B}$ be symmetric.
    If $ \tilde{B} = \sum_{i=1}^r \lambda_i v_i v_i^\top $ is the eigendecomposition of $ \tilde{B} $,
    then 
    \begin{equation}
    S = \sum\limits_{i=1}^r \lambda_i (\tilde{\Phi} v_i \otimes \tilde{\Phi} v_i)
    \end{equation} is the eigendecomposition of $S$.
\end{corollary}

In particular, the matrix $ \tilde{B} $ and the operator $ S $ share the same singular values (or eigenvalues, respectively) potentially up to zero.
In practice, computing the singular value decomposition of $ S $ by this approach needs two kernel QR decompositions (which numerically results in Cholesky decompositions of the Gram matrices), inversions of the triangular matrices $R_\Psi$ and $R_\Phi$ and the final decomposition of $\tilde{B}$.
For the eigendecomposition we need a single kernel QR decomposition and inversion before performing the eigendecomposition of $ \tilde{B}$.
Since this may numerically be costly, we give an overview of how eigendecompositions and singular value decompositions of empirical RKHS
operators can be performed by solving a single related auxiliary problem.

\begin{remark}
    Representation~\eqref{eq:QR_operator} makes it possible to compute
    classical matrix decompositions such as Schur decompositions, LU-type decompositions, or polar decompositions on $\tilde{B}$ and obtain a corresponding decomposition of the operator $S$.
    Note however that when $S$ approximates an operator for
    $n,m \to \infty$, it is not necessarily given that these empirical decompositions of $ S $ converge to a meaningful infinite-rank concept that is equivalent. For the eigendecomposition and the
    singular value decomposition, this reduces
    to classical operator perturbation theory~\cite{Kato80}.
\end{remark}

\subsection{Eigendecomposition via auxiliary problem}
\label{sec:eigendecomposition}

The eigendecomposition of RKHS operators via an auxiliary problem was first considered in \cite{KSM19}. For the sake of completeness, we will briefly recapitulate the main result and derive additional properties. For the eigendecomposition, we again require the operator to be a mapping from $ \mathscr{H} $ to itself. For this section, we  define a new feature matrix by $ \Upsilon = [\phi(x'_1), \dots, \phi(x'_m)] $. Note that the sizes of $ \Phi $ and $ \Upsilon $ have to be identical.

\begin{proposition}[cf.~\cite{KSM19}] \label{prop:eigendecomposition}
    Let $ S \colon \mathscr{H} \to \mathscr{H} $ with $ S = \Upsilon B \Phi^\top $ and $ B \in \R^{m \times m} $ be an empirical RKHS operator. Then the following
    statements hold:
    \begin{enumerate}[label=(\roman*)]
        \item If $ \lambda $ is an eigenvalue of $ B \Phi^\top \Upsilon \in \R^{m \times m} $ with corresponding eigenvector $ \mathbf{w} \in \R^m $, then $ \Upsilon \mathbf{w} \in \mathscr{H} $ is an eigenfunction of $ S $ corresponding to $ \lambda $.
        \item Conversely, if $ \lambda \neq 0 $ is an eigenvalue of $ S $ corresponding to the eigenfunction $ v \in \mathscr{H} $, then $ B \Phi^\top v \in \R^m $ is an eigenvector of $ B \Phi^\top \Upsilon \in \R^{m \times m} $ corresponding to the eigenvalue~$ \lambda $.
    \end{enumerate}
    In particular, the operator $ S $ and the matrix $ B \Phi^\top \Upsilon $ share the same nonzero eigenvalues.
\end{proposition}

\begin{proof}
    For the sake of completeness, we briefly reproduce the gist of the proof.
    \begin{enumerate}[label=(\roman*), wide, labelwidth=!, labelindent=0pt]
        \item Let $ \mathbf{w} \in \R^m $ be an eigenvector of the matrix $ B \Phi^\top \Upsilon $ corresponding to the eigenvalue $ \lambda $. Using the associativity of feature matrix multiplication and kernel evaluation, we have
        \begin{equation*}
        S (\Upsilon \mathbf{w}) = \Upsilon (B \Phi^\top \Upsilon \mathbf{w}) = \lambda \Upsilon \mathbf{w}.
        \end{equation*}
        Furthermore, since $ \mathbf{w} \neq 0 \in \R^m $ and the elements in $ \Upsilon $ are linearly independent, we have $ \Upsilon \mathbf{w} \neq 0 \in \mathscr{H} $. Therefore, $ \Upsilon \mathbf{w} $ is an eigenfunction of $ S $ corresponding to $ \lambda $.
        \item Let $ v $ be an eigenfunction of $ S $ associated with the eigenvalue $ \lambda \neq 0 $. By assumption, we then have
        \begin{equation*}
        \Upsilon B \Phi^\top v = \lambda v.
        \end{equation*}
        By ``multiplying'' both sides from the left with $ B \Phi^\top $ and using the associativity of the feature matrix notation, we obtain
        \begin{equation*}
        (B \Phi^\top \Upsilon) B \Phi^\top v = \lambda B \Phi^\top v.
        \end{equation*}
        Furthermore, $ B \Phi^\top v $ cannot be the zero vector in $ \R^m $ as we would have $ \Upsilon (B \Phi^\top v) = S v = 0 \neq \lambda v $ otherwise since $ \lambda $ was assumed to be a nonzero eigenvalue. Therefore, $ B \Phi^\top v $ is an eigenvector of the matrix $ B \Phi^\top \Upsilon $.
    \end{enumerate}
\end{proof}

\begin{remark} \label{rem:eigendecomposition_reformulation}
    Eigenfunctions of empirical RKHS operators may be expressed as a linear combination of elements contained in the feature matrices. However, there exist other formulations of this result \cite{KSM19}. We can, for instance, define the alternative auxiliary problem
    \begin{equation*}
    \Phi^\top \Upsilon B \mathbf{w} = \lambda \mathbf{w}.
    \end{equation*}
    For eigenvalues $ \lambda $ and eigenvectors $ \mathbf{w} \in \R^m $ satisfying this equation, we see that $ \Upsilon B v \in \mathscr{H} $ is an eigenfunction of $ S $. Conversely, for eigenvalues $ \lambda \neq 0 $ and eigenfunctions $ v \in \mathscr{H} $ of $ S $, the auxiliary matrix has the eigenvector $ \Phi^\top v \in \R^m $.
\end{remark}

\begin{example} \label{ex:C_xx}
The eigendecomposition of RKHS operators can be used to obtain an approximation of the Mercer feature space representation of a given kernel. Let us consider the domain $ \inspace = [-2, 2] \times [-2, 2] $ equipped with the Lebesgue measure and the kernel $ k(x, x^\prime) = \left(1 + x^\top x^\prime\right)^2 $. The eigenvalues and eigenfunctions of the integral operator $ \mathcal{E}_k $ defined by
\begin{equation*}
\ebd[k] f(x) = \int k(x, x^\prime) \tsp f(x^\prime) \tsp \dd \mu(x^\prime)
\end{equation*}
are given by
\begin{alignat*}{4}
\lambda_1 &= \tfrac{269 + \sqrt{60841}}{90} \approx 5.72, & \quad e_1(x) &= c_1 \left(\tfrac{-179 + \sqrt{60841}}{120} + x_1^2 + x_2^2 \right), \\
\lambda_2 &= \tfrac{32}{9} \approx 3.56,                  & \quad e_2(x) &= c_2 \tsp x_1 \tsp x_2, \\[1ex]
\lambda_3 &= \tfrac{8}{3} \approx 2.67,                   & \quad e_3(x) &= c_3 \tsp x_1, \\[1ex]
\lambda_4 &= \tfrac{8}{3} \approx 2.67,                   & \quad e_4(x) &= c_4 \tsp x_2, \\[1ex]
\lambda_5 &= \tfrac{64}{45} \approx 1.42,                 & \quad e_5(x) &= c_5 \left(x_1^2 - x_2^2\right), \\[1ex]
\lambda_6 &= \tfrac{269 - \sqrt{60841}}{90} \approx 0.24, & \quad e_6(x) &= c_6 \left(\tfrac{-179 - \sqrt{60841}}{120} + x_1^2 + x_2^2 \right),
\end{alignat*}
where $ c_1, \dots, c_6 $ are normalization constants so that $ \norm{e_i}_\mu = 1 $. Defining $ \phi = [\phi_1, \dots, \phi_6]^\top $, with $ \phi_i = \sqrt{\lambda_i} \tsp e_i $, we thus obtain the Mercer feature space representation of the kernel, i.e., $ k(x, x^\prime) = \innerprod{\phi(x)}{\phi(x^\prime)} $. Here, $ \innerprod{\cdot}{\cdot} $ denotes the standard inner product in $ \R^6 $. For $ f \in \mathscr{H} $, it holds that $ \mathcal{E}_k f = \cov[XX] f $, where $ \cov[XX] $ is the covariance operator.\footnote{For a detailed introduction of covariance and cross-covariance operators, see Section~\ref{sec:Applications}.} We now compute eigenfunctions of its empirical estimate $ \ecov[XX] $ with the aid of the methods described above. That is, $ B = \frac{1}{m} I_m $. Drawing $ m = 5000 $ test points from the uniform distribution on $ \inspace $, we obtain the eigenvalues and (properly normalized) eigenfunctions shown in Figure~\ref{fig:C_xx}. The eigenfunctions are virtually indistinguishable from the analytically computed ones. Note that the eigenspace corresponding to the eigenvalues $ \lambda_{3} $ and $ \lambda_4 $ is only determined up to basis rotations. The eigenvalues $ \lambda_i $ for $ i > 6 $ are numerically zero. \exampleSymbol

\begin{figure}[tbh]
    \centering
    \begin{minipage}{0.32\textwidth}
        \centering
        \subfiguretitle{(a) $ \lambda_1 = 5.78 $}
        \includegraphics[width=0.9\textwidth]{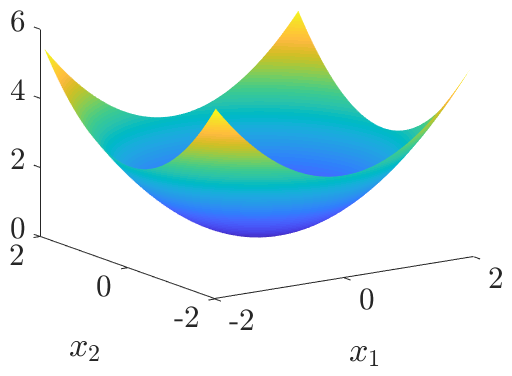} \\
        \subfiguretitle{(b) $ \lambda_2 = 3.56 $}
        \includegraphics[width=0.9\textwidth]{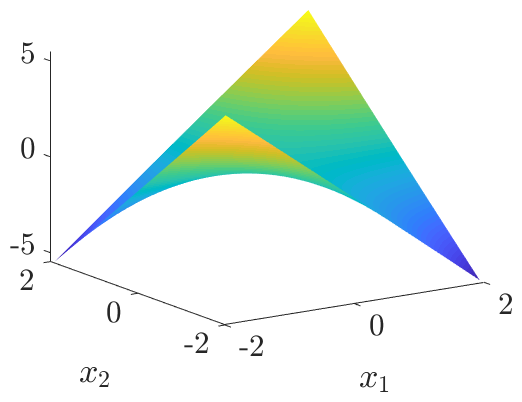}
    \end{minipage}
    \begin{minipage}{0.32\textwidth}
        \centering
        \subfiguretitle{(c) $ \lambda_3 = 2.69 $}
        \includegraphics[width=0.9\textwidth]{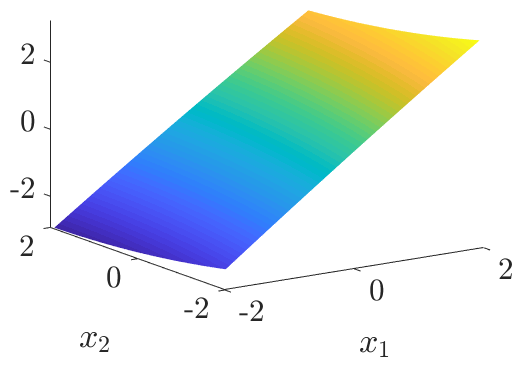} \\
        \subfiguretitle{(d) $ \lambda_4 = 2.66 $}
        \includegraphics[width=0.9\textwidth]{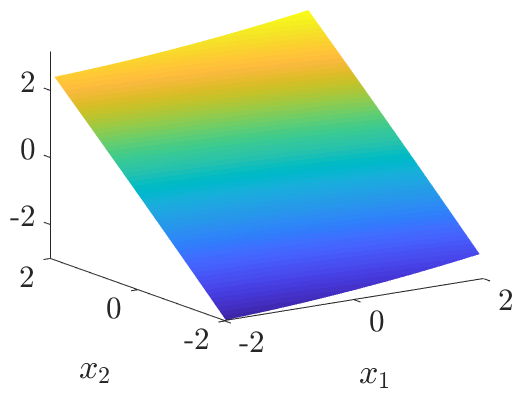}
    \end{minipage}
    \begin{minipage}{0.32\textwidth}
        \centering
        \subfiguretitle{(e) $ \lambda_5 = 1.46 $}
        \includegraphics[width=0.9\textwidth]{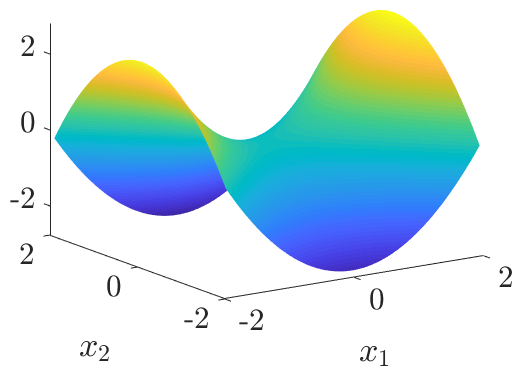} \\
        \subfiguretitle{(f) $ \lambda_6 = 0.24 $}
        \includegraphics[width=0.9\textwidth]{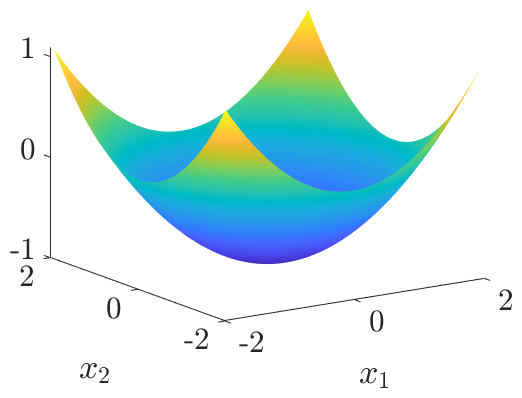}
    \end{minipage}
    \caption{Numerically computed eigenvalues and eigenfunctions of $ \ecov[XX] $ associated with the second-order polynomial kernel on $ \inspace = [-2, 2] \times [-2, 2] $.}
    \label{fig:C_xx}
\end{figure}

\end{example}

While we need the assumption that the eigenvalue $ \lambda $ of $ S $ is nonzero to infer the eigenvector of the auxiliary matrix from the eigenfunction from $ S $, this assumption is not needed the other way around.
This has the simple explanation that a rank deficiency of $ B $ always introduces a rank deficiency to $ S = \Upsilon B \Phi^\top $. On the other hand, if $ \mathscr{H} $ is infinite-dimensional, $ S $ as a finite-rank operator \emph{always} has a natural rank deficiency, even when $ B $ has full rank. In this case, $ S $ has the eigenvalue $ 0 $ while $ B $ does not.

In order to use Proposition \ref{prop:eigendecomposition} as a consistent tool to compute eigenfunctions of RKHS operators, we must ensure that all eigenfunctions corresponding to nonzero eigenvalues of empirical RKHS operators can be computed. In particular, we have to be certain that eigenvalues with a higher geometric multiplicity allow to capture a full set of linearly independent basis eigenfunctions in the associated eigenspace.

\begin{lemma} \label{lem:independent_eigenfunctions}
    Let $ S \colon \mathscr{H} \to \mathscr{H} $ with $ S = \Upsilon B \Phi^\top $ be an empirical RKHS operator. Then it holds:
    \begin{enumerate}[label=(\roman*)]
        \item If $ \mathbf{w}_1 \in \R^m $ and $ \mathbf{w}_2 \in \R^m $ are linearly independent eigenvectors of $ B \Phi^\top \Upsilon $, then $ \Upsilon \mathbf{w}_1 \in \mathscr{H} $ and $ \Upsilon \mathbf{w}_2 \in \mathscr{H} $ are linearly independent eigenfunctions of $ S $.
        \item If $ v_1 $ and $ v_2 $ are linearly independent eigenfunctions belonging to the eigenvalue $ \lambda \neq 0 $ of $ S $, then $ B \Phi^\top v_1 \in \R^m $ and $ B \Phi^\top v_2 \in \R^m $ are linearly independent eigenvectors of $ B \Phi^\top \Upsilon$.
    \end{enumerate}
    In particular, if $ \lambda \neq 0 $, then we have $ \mdim \mker (B \Phi^\top \Upsilon - \lambda \id_m) = \mdim \mker (S - \lambda \idop_\mathscr{H}) $.
\end{lemma}

\begin{proof}
    The eigenvalue-eigenfunction correspondence is covered in Proposition \ref{prop:eigendecomposition}, it therefore remains to check the linear independence in statements (i) and (ii). Part (i) follows from Remark \ref{rem:independent_features}. We show part (ii) by contradiction: Let $ v_1 $ and $ v_2 $ be linearly independent eigenfunctions associated with the eigenvalue $ \lambda \neq 0 $ of $ S $. Then assume for some $ \alpha \neq 0 \in \R $, we have $ B \Phi^\top v_1 = \alpha B \Phi^\top v_2 $. Applying $ \Upsilon $ from the left to both sides, we obtain
    \begin{equation*}
    \Upsilon B \Phi^\top v_1 = S v_1 = \lambda v_1 =
    \alpha \lambda v_2 = \alpha S v_2 = \Upsilon \alpha B \Phi^\top v_2,
    \end{equation*}
    which contradicts the linear independence of $ v_1 $ and $ v_2 $. Therefore, $ B \Phi^\top v_1 $ and $ B \Phi^\top v_2$ have to be linearly independent in $ \R^m $.
    
    From (i) and (ii), we can directly infer $ \mdim \mker (B \Phi^\top \Upsilon - \lambda \id_m) = \mdim \mker (S - \lambda \idop_\mathscr{H}) $ by contradiction: Let $ \lambda \neq 0 $ be an eigenvalue of $ S $ and $ B \Phi^\top \Upsilon$. We assume that $ \mdim \mker (B \Phi^\top \Upsilon - \lambda I_m) > \mdim \mker (S - \lambda \idop_\mathscr{H}) $. This implies that there exist two eigenvectors $ \mathbf{w}_1, \mathbf{w}_2 \in \R^m $ of $ B \Phi^\top \Upsilon $ that generate two linearly dependent eigenfunctions $ \Upsilon \mathbf{w}_1, \Upsilon \mathbf{w}_2 \in \mathscr{H} $, contradicting statement (i). Hence, we must have $ \mdim \mker (B \Phi^\top \Upsilon - \lambda I_m) \leq \mdim \mker (S - \lambda \idop_\mathscr{H})$.  Analogously, applying the same logic to statement (ii), we obtain $ \mdim \mker (B \Phi^\top \Upsilon - \lambda I_m) \geq \mdim \mker (S - \lambda \idop_\mathscr{H}) $, which concludes the proof.
\end{proof}

\begin{corollary} \label{cor:auxiliary}
    If $ S = \Upsilon B \Phi^\top $ is an empirical RKHS operator and $ \lambda \in \R $ is nonzero, it holds that
    $ \{ \Upsilon \mathbf{w} \mid B \Phi^\top \Upsilon \mathbf{w} = \lambda \mathbf{w} \} =  \mker (S - \lambda \id_\mathscr{H}) $.
\end{corollary}

The corollary justifies to refer to the eigenvalue problems $ S v = \lambda v $ as \emph{primal problem} and $ B \Phi^\top \Upsilon \mathbf{w} = \lambda \mathbf{w} $ as \emph{auxiliary problem}, respectively.

\subsection{Singular value decomposition via auxiliary problem}
\label{sec:SVD}

We have seen that we can compute eigenfunctions corresponding to nonzero eigenvalues of empirical RKHS operators. This can be extended in a straightforward fashion to the singular value decomposition of such operators.

\subsubsection{Standard derivation}

We apply the eigendecomposition to the self-adjoint operator $ S^*S $ to obtain the singular value decomposition of $ S $.

\begin{proposition} \label{prop:svd}
Let $ S \colon \mathscr{H} \to \mathscr{F} $ with $ S = \Psi B \Phi^\top $  be an empirical RKHS operator, where $ \Phi = [\phi(x_1), \dots, \phi(x_m)] $, $ \Psi = [\psi(y_1), \dots, \psi(y_n)]$, and $ B \in \R^{n \times m} $. 
Assume that the multiplicity of each singular value of $S$ is $1$.
Then the SVD of $ S $ is given by
\begin{equation*}
    S = \sum_{i=1}^r \lambda_i^{1/2} (u_i \otimes v_i),
\end{equation*}
where
\begin{align*}
  v_i &:= (\mathbf{w}_i^\top G_\Phi \mathbf{w}_i)^{-1/2} \, \Phi \mathbf{w}_i, \\
  u_i &:= \lambda_i^{-1/2} S v_i,
\end{align*}
with the nonzero eigenvalues $ \lambda_1, \dots, \lambda_r \in \R $
of the matrix
\begin{align*}
  M G_{\Phi} \in \R^{m \times m} \quad \textit{with} \quad M:= B^\top G_{\Psi} B \in \R^{m \times m}
\end{align*}
counted with their multiplicities and corresponding eigenvectors $ \mathbf{w}_1, \dots, \mathbf{w}_r \in \R^m $.
\end{proposition}

\begin{proof}
Using Proposition~\ref{prop:rkhs_operator_properties}, the operator
\begin{equation*}
    S^*S = \Phi (B^\top G_{\Psi} B) \Phi^\top = \Phi M \Phi^\top
\end{equation*}
is an empirical RKHS operator on $ \mathscr{H} $. Naturally, $ S^*S $ is also positive and self-adjoint. We apply Corollary~\ref{cor:auxiliary} to calculate the normalized eigenfunctions
\begin{equation*}
    v_i := \norm{\Phi \mathbf{w}_i}_{\mathscr{H}}^{-1} \Phi \mathbf{w}_i
         = (\mathbf{w}_i^\top G_{\Phi} \mathbf{w}_i)^{-1/2} \, \Phi \mathbf{w}_i
\end{equation*}
of $ S^*S $ by means of the auxiliary problem
\begin{equation*}
    M G_{\Phi} \mathbf{w}_i = \lambda_i \mathbf{w}_i, \quad \mathbf{w}_i \in \R^m,
\end{equation*}
for nonzero eigenvalues $ \lambda_i $. We use Lemma~\ref{lem:eigendecomposition_svd} to establish the connection between the eigenfunctions of $ S^*S $ and singular functions of $ S $ and obtain the desired form for the SVD of $ S $.
\end{proof}

\begin{remark}
    Whenever the operator $S$ possesses singular values with multiplicities larger than
    $1$, a Gram-Schmidt procedure may need to be applied to the resulting 
    singular functions in order to ensure that they form an orthonormal system in the
    corresponding eigenspaces of $S^*S$ and $SS^*$.
\end{remark}

\begin{remark}
As described in Remark~\ref{rem:eigendecomposition_reformulation}, several different auxiliary problems to compute the eigendecomposition of $ S^*S $ can be derived. As a result, we can reformulate the calculation of the SVD of $ S $ for every possible auxiliary problem.
\end{remark}

\begin{example} \label{ex:C_yx}
We define a probability density on $ \R^2 $ by
\begin{equation*}
    p(x, y) = \frac{1}{2} \big(p_1(x) \tsp p_2(y) + p_2(x) \tsp p_1(y)\big),
\end{equation*}
with
\begin{equation*}
    p_1(x) = \tfrac{1}{\sqrt{2 \pi \rho^2}} \tsp e^{-\frac{(x-1)^2}{2 \rho^2 \vphantom{X^X}}}
    \quad \text{and} \quad
    p_2(x) = \tfrac{1}{\sqrt{2 \pi \rho^2}} \tsp e^{-\frac{(x+1)^2}{2 \rho^2 \vphantom{X^X}}},
\end{equation*}
see Figure~\ref{fig:C_yx}(a), and draw $ m = n = 10000 $ test points $ (x_i, y_i) $ from this density as shown in Figure~\ref{fig:C_yx}(b). Let us now compute the singular value decomposition of $ \ecov[YX] = \frac{1}{m} \Psi \tsp \Phi^\top $, i.e., $ B = \frac{1}{m} I_m $. That is, we have to compute the eigenvalues and eigenvectors of the auxiliary matrix $ \frac{1}{m^2} \gram[\Psi] \tsp \gram[\Phi] $. Using the normalized Gaussian kernel with bandwidth $ 0.1 $ results in singular values $ \sigma_1 \approx 0.47 $ and $ \sigma_2 \approx 0.43 $ and the corresponding right and left singular functions displayed in Figure~\ref{fig:C_yx}(c) and Figure~\ref{fig:C_yx}(d). The subsequent singular values are close to zero. Thus, we can approximate $ \ecov[YX] $ by a rank-two operator of the form $ \ecov[YX] \approx \sigma_1 (u_1 \otimes v_1) + \sigma_2 (u_2 \otimes v_2) $, see also Figure~\ref{fig:C_yx}(e) and Figure~\ref{fig:C_yx}(f). This is due to the decomposability of the probability density $ p(x, y) $. \exampleSymbol

\begin{figure}[tbh]
    \centering
    \begin{minipage}{0.32\textwidth}
        \centering
        \subfiguretitle{(a)}
        \includegraphics[width=\textwidth]{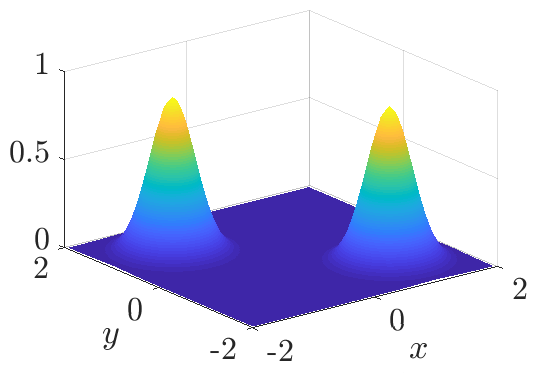} \\
        \subfiguretitle{(b)}
        \includegraphics[width=\textwidth]{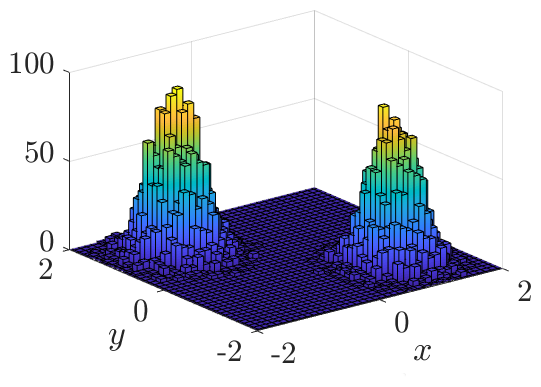}
    \end{minipage}
    \begin{minipage}{0.32\textwidth}
        \centering
        \subfiguretitle{(c)}
        \includegraphics[width=0.9\textwidth]{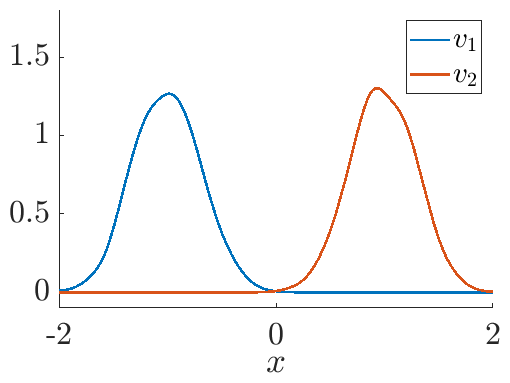} \\
        \subfiguretitle{(d)}
        \includegraphics[width=0.9\textwidth]{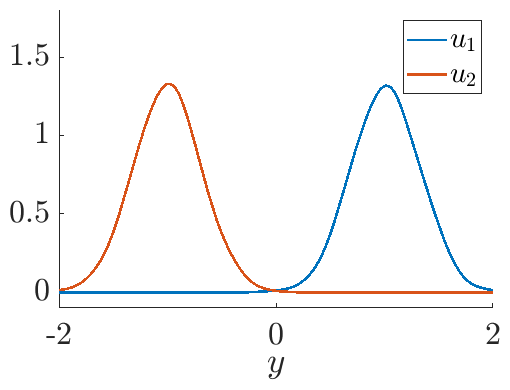}
    \end{minipage}
    \begin{minipage}{0.32\textwidth}
        \centering
        \subfiguretitle{(e)}
        \includegraphics[width=\textwidth]{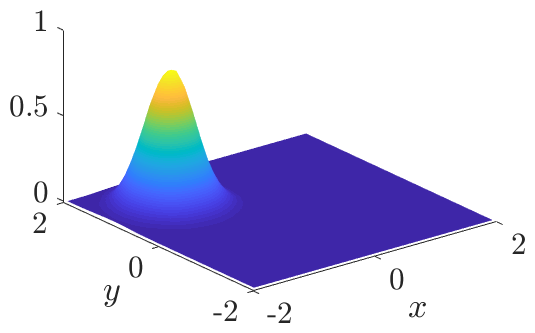} \\
        \subfiguretitle{(f)}
        \includegraphics[width=\textwidth]{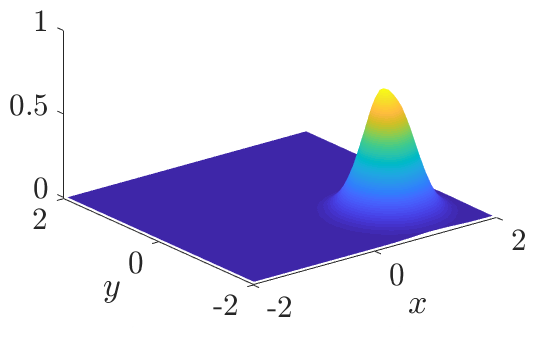}
    \end{minipage}
    \caption{Numerically computed singular value decomposition of $ \ecov[YX] $. (a) Joint probability density $ p(x, y) $. (b) Histogram of the 10000 sampled data points. (c) First two right singular functions. (d) First two left singular functions. (e) $ \sigma_1 (u_1 \otimes v_1) $. (f) $ \sigma_2 (u_2 \otimes v_2) $.}
    \label{fig:C_yx}
\end{figure}

\end{example}

With the aid of the singular value decomposition, we are now, for instance, able to compute low-rank approximations of RKHS operators---e.g., to obtain more compact and smoother representations---or their pseudoinverses. This will be described below. First, however, we show an alternative derivation of the decomposition. Proposition~\ref{prop:svd} gives a numerically computable form of the SVD of the empirical RKHS operator~$ S $. Since the auxiliary problem of the eigendecomposition of $ S^*S $ involves several matrix multiplications, the problem might become ill-conditioned.

\subsubsection{Block-operator formulation}

We now employ the relationship described in Corollary~\ref{cor:block_operator_svd} between the SVD of the empirical RKHS operator $ S \colon \mathscr{H} \rightarrow \mathscr{F} $ and the eigendecomposition of the block-operator $ T \colon \mathscr{F} \oplus \mathscr{H} \rightarrow \mathscr{F} \oplus \mathscr{H} $, with
$ (f,h) \mapsto (Sh,S^* f) $.

\begin{theorem}
  The SVD of the empirical RKHS operator $S = \Psi B \Phi^\top$ is given by
  \begin{equation*}
    S = \sum\limits_{i \in I}^r
    \sigma_i \left[ \left(\norm{\Psi \mathbf{w}_i}_{\mathscr{F}}^{-1} \Psi \mathbf{w}_i \right) \otimes \left(\norm{\Phi \mathbf{z}_i}_{\mathscr{H}}^{-1} \Phi \mathbf{z}_i \right) \right],
  \end{equation*}
  where $\sigma_i$ are the strictly positive eigenvalues and
  $
    \begin{bsmallmatrix}
      \mathbf{w}_i \\ \mathbf{z}_i
    \end{bsmallmatrix} \in \R^{n + m}
  $ the corresponding eigenvectors of the auxiliary matrix
  \begin{equation} \label{eq:block_aux}
    \begin{bmatrix}
      0 & B G_{\Phi} \\
      B^\top G_{\Psi} & 0
    \end{bmatrix} \in \R^{(n + m) \times (n + m)}.
  \end{equation}
\end{theorem}

\begin{proof}
  The operator $ T $ defined above can be written in block form as
  \begin{equation} \label{eq:rkhs_blockoperator}
    T \begin{bmatrix}
        f \\ h
      \end{bmatrix} =
    \begin{bmatrix}
       & S \\
       S^* &
    \end{bmatrix}
    \begin{bmatrix}
      f \\ h
    \end{bmatrix} =
    \begin{bmatrix}
      S h \\
      S^* f
    \end{bmatrix}.
  \end{equation}
  By introducing the block feature matrix $ \Lambda := \left[\, \Psi \; \Phi \,\right] $, we may rewrite \eqref{eq:rkhs_blockoperator} as the empirical RKHS operator
  \begin{align*}
    \Lambda
    \begin{bmatrix}
       0 & B \\
       B^\top & 0
    \end{bmatrix}
    \Lambda^\top.
  \end{align*}
  Invoking Corollary \ref{cor:auxiliary} yields the auxiliary problem
  \begin{align*}
    \begin{bmatrix}
      0 & B \\
      B^\top & 0
    \end{bmatrix}
    \Lambda^\top \Lambda
    =
    \begin{bmatrix}
        0 & B \\
        B^\top & 0
      \end{bmatrix}
    \begin{bmatrix}
        G_{\Psi} & 0 \\
        0 & G_{\Phi}
    \end{bmatrix}
    =
    \begin{bmatrix}
      0 & B G_\Phi \\
      B^\top G_\Psi & 0
    \end{bmatrix} \in \R^{(n + m) \times (n + m)}
  \end{align*}
  for the eigendecomposition of $T$. We again emphasize
  that the block-operator notation has to be used with
  caution since $\mathscr{F} \oplus \mathscr{H}$ is an external direct sum.
  We use Corollary~\ref{cor:block_operator_svd} to obtain the SVD of $S$ from the eigendecomposition of $T$.
\end{proof}

\begin{remark}
In matrix analysis and numerical linear algebra, one often computes the SVD of a matrix $A \in \R^{n \times m}$ through an eigendecomposition of the matrix $\begin{bsmallmatrix} 0 & A \\ A^\top & 0 \end{bsmallmatrix}$. This leads to a symmetric problem, usually simplifying iterative SVD schemes~\cite{Golub13}. The auxiliary problem~\eqref{eq:block_aux}, however, is in general not symmetric.
\end{remark}

\section{Applications}
\label{sec:Applications}

In this section, we describe different operators of the form $ S = \Psi B \Phi^\top $ or $ S = \Phi B \Psi^\top $, respectively, and potential applications. All of the presented examples are empirical estimates of Hilbert--Schmidt RKHS operators. Therefore, the SVD of the given empirical RKHS operators converges to the SVD of their analytical counterparts.
For results concerning the convergence and consistency of the estimators, we refer to~\cite{Gruen12,Fukumizu13:KBR,Fukumizu15:NBI,Park2020MeasureTheoretic}.
Note that in practice the examples below may bear additional
challenges such as ill-posed inverse problems and regularization of compact operators, which we will not examine in detail.
We will also not cover details such as measurability of feature maps and properties of related integral operators in what follows.
For these details, the reader may consult, for example,~\cite{StCh08}.

\subsection{Low-rank approximation, pseudoinverse and optimization}

With the aid of the SVD it is now also possible to compute low-rank approximations of RKHS operators. This well-known result is called \emph{Eckart--Young theorem} or \emph{Eckart--Young--Mirsky theorem}, stating that the finite-rank operator given by the truncated SVD
\begin{equation*}
A_k := \sum_{i=1}^k \sigma_i (u_i \otimes v_i)
\end{equation*}
satisfies the optimality property
\begin{equation*}
A_k = \argmin_{\mathclap{\rank(B) = k}} \norm{A-B}_\HS,
\end{equation*}
see~\cite{Eubank15} for details. Another application is the computation of the
(not necessarily globally defined) \emph{pseudoinverse} or \emph{Moore--Penrose inverse}~\cite{EHN96} of operators, defined as $ A^+ \colon \mathscr{F} \supseteq \mathrm{dom}(A^+) \rightarrow \mathscr{H} $, with
\begin{equation*}
A^+ = \sum\limits_{i = 1}^r \sigma_i^{-1} (v_i \otimes u_i).
\end{equation*}
We can thus obtain the solution $ x \in H $ of the---not necessarily well-posed---inverse problem $ A x = y $ for $ y \in \mathscr{F} $ through the Moore--Penrose pseudoinverse, i.e.,
\begin{equation*}
A^+ y = \argmin_{\mathclap{x \in \mathscr{H}}} \norm{Ax - y}_\mathscr{F},
\end{equation*}
where $ A^+ y $ in $ \mathscr{H} $ is the unique minimizer with minimal norm. For the 
connection to regularized least-squares problems and the theory of inverse problems, see~\cite{EHN96}.

\subsection{Kernel covariance and cross-covariance operator}

The \emph{kernel covariance operator} $ \cov[XX] \colon \mathscr{H} \to \mathscr{H} $ and the \emph{kernel cross-covariance operator}~\cite{Baker1973} $ \cov[YX] \colon \mathscr{H} \to \mathscr{F} $ are defined by
\begin{alignat*}{4}
    \cov[XX] &= \int \phi(X) \otimes \phi(X) \tsp \dd \pp{P}(X)
             &&= \mathbb{E}_{\scriptscriptstyle X}[\phi(X) \otimes \phi(X)], \\
    \cov[YX] &= \int \psi(Y) \otimes \phi(X) \tsp \dd \pp{P}(Y,X)
             &&= \mathbb{E}_{\scriptscriptstyle \mathit{YX}}[\psi(Y) \otimes \phi(X)],
\end{alignat*}
assuming that the second moments (in the Bochner integral sense) of the embedded random
variables $X,Y$ exist.
Kernel (cross-)covariance operators can be regarded as generalizations of (cross\nobreakdash-)covariance matrices and are frequently used in nonparametric statistical methods, see \cite{MFSS17} for an overview. Given training data $ \mathbb{D}_{\scriptscriptstyle XY} = \{(x_1, y_1), \dots, (x_n, y_n)\} $ drawn i.i.d.\ from the joint probability distribution $ \pp{P}(X, Y) $, we can estimate these operators by
\begin{equation*}
    \ecov[XX] = \frac{1}{n} \sum_{i=1}^n \phi(x_i) \otimes \phi(x_i)
              = \frac{1}{n} \Phi \Phi^\top
    \quad \text{and} \quad
    \ecov[YX] = \frac{1}{n} \sum_{i=1}^n \psi(y_i)\otimes\phi(x_i)
              = \frac{1}{n} \Psi \Phi^\top.
\end{equation*}
Thus, $ \ecov[XX] $ and $ \ecov[YX] $ are empirical RKHS operators with $ B = \frac{1}{n} I_n $, where $ \Psi = \Phi $ for $ \ecov[XX] $. Decompositions of these operators are demonstrated in Example~\ref{ex:C_xx} and Example~\ref{ex:C_yx}, respectively, where we show that we can compute approximations of the Mercer feature space and obtain low-rank approximations of operators.

\subsection{Conditional mean embedding}

The conditional mean embedding is an extension of the mean embedding framework to conditional probability distributions. Under some technical assumptions, the RKHS embedding of a 
conditional distribution can be represented as a linear operator~\cite{SHSF09}.
We will not cover the technical details here and refer the reader to~\cite{Klebanov2019rigorous} for
the mathematical background. We note that alternative interpretations of
the conditional mean embedding exist in a least-squares context which needs less assumptions
than the operator-theoretic formulation~\cite{Gruen12,Park2020MeasureTheoretic}.
\begin{remark}
For simplicity, we write $\cov[XX]^{-1}$ for the inverse covariance operator
in what follows. However, note that $\cov[XX]^{-1}$ does in general not exist as a globally defined bounded operator --in practice, a Tikhonov-regularized inverse (i.e., $(\cov[XX] + \epsilon \mathrm{Id})^{-1}$ for some $\epsilon>0$) is usually considered instead
(see~\cite{EHN96} for details), leading to regularized matrices in the empirical versions.
\end{remark}

The conditional mean embedding operator of $ \pp{P}(Y \mid X) $ is given by
\begin{equation*}
    \cme = \cov[YX] \tsp \cov[XX]^{-1}.
\end{equation*}
Note that when the joint distribution $\pp{P}(X,Y)$ and hence
$\cov[XX]$ and $\cov[YX]$ are unknown,
we can not compute $\cme$ directly.
However, if the training data $ \mathbb{D}_{\scriptscriptstyle XY} =
\{(x_1, y_1), \dots, (x_n, y_n)\} $ is drawn i.i.d.\ from the probability distribution $ \pp{P}(X,Y) $,
it can be estimated as
\begin{equation*}
    \ecme = \Psi \tsp \gram[\phi]^{-1} \tsp \Phi^\top.
\end{equation*}
This is an empirical RKHS operator, where $ B = \gram[\phi]^{-1} $.
The conditional mean operator is often used for nonparametric models, for example in state-space models~\cite{SHSF09}, filtering and Bayesian inference~\cite{Fukumizu13:KBR,Fukumizu15:NBI}, reinforcement learning~\cite{LeverEtAl:CCME17,Stafford:ACCME18,Gebhardt2019Robot},
and density estimation~\cite{SMKM20:CDO}.

\subsection{Kernel transfer operators}

Transfer operators such as the Perron--Frobenius operator~$ \pf $ and Koopman operator~$ \ko $ are frequently used for the analysis of the global dynamics of molecular dynamics and fluid dynamics problems but also for model reduction and control~\cite{KNKWKSN18}. Approximations of these operators in RHKSs are strongly related to the conditional mean embedding framework~\cite{KSM19}.
The kernel-based variants $ \pf[k] $ and $ \ko[k] $ are defined by
\begin{equation*}
    \pf[k] = \cov[XX]^{-1} \tsp \cov[YX]
    \quad \text{and} \quad
    \ko[k] = \cov[XX]^{-1} \tsp \cov[XY],
\end{equation*}
where we consider a (stochastic) dynamical system $X = (X_t)_{t \in T}$
and a time-lagged version of itself $Y = (X_{t + \tau})_{t \in T}$ for a fixed time lag $\tau$.
The empirical estimates of $ \pf[k] $ and $ \ko[k] $ are given by
\begin{equation*}
    \epf[k] = \Psi \tsp \gram[\Phi \Psi]^{-1} \tsp \gram[\Phi]^{-1} \tsp \gram[\Phi \Psi] \tsp \Phi^\top
    \quad \text{and} \quad
    \eko[k] = \Phi \tsp \gram[\Phi]^{-1} \tsp \Psi^\top.
\end{equation*}
Here, we use the feature matrices
\begin{equation*}
\Phi := [\phi(x_1), \dots, \phi(x_m)] \quad \text{and} \quad \Psi := [\phi(y_1), \dots, \phi(y_n)]
\end{equation*}
with data $x_i$ and $y_i = \Xi^\tau(x_i)$, where $\Xi^\tau$ denotes the flow map
associated with the dynamical system $X$ with time step $\tau$.
Note that in particular $\mathscr{H} = \mathscr{F}$.
Both operators $ \pf[k] $ and $ \ko[k] $ can be written as empirical RKHS operators, with $ B = \gram[\Phi \Psi]^{-1} \tsp \gram[\Phi]^{-1} \tsp \gram[\Phi \Psi] $ and $ B = \gram[\Phi]^{-1} $, respectively, where $ \gram[\Phi \Psi] = \Phi^\top \Psi $ is a time-lagged Gram matrix. Examples pertaining to the eigendecomposition of kernel transfer operators associated with molecular dynamics and fluid dynamics problems as well as text and video data can be found in \cite{KSM19}. The eigenfunctions and corresponding eigenvalues of kernel transfer operators contain information about the dominant slow dynamics and their implied time-scales. Moreover, the singular value decomposition of kernel transfer operators is known to be connected to \emph{kernel canonical correlation analysis}~\cite{MelzerEtAl:CCA} and the detection of coherent sets in dynamical systems~\cite{KHM19}.
In particular, the singular value decomposition of the operator 
\begin{equation*}
S := \ecov[YY]^{-1/2} \ecov[YX] \tsp \ecov[XX]^{-1/2}
\end{equation*}
solves the kernel CCA problem. This operator can be written
as
\begin{equation*}
    S = \Psi B \Phi^\top, 
\end{equation*}
where $B = \gram[\Psi]^{-1/2} \gram[\Phi]^{-1/2}$. For the derivation, see Appendix~\ref{app:empirical_cca}.
We will give an example in the context of coherent sets to illustrate potential applications.

\begin{example}
Let us consider the well-known periodically driven double gyre flow
\begin{align*}
    \dot{x}_1 &= -\pi A \sin(\pi f(x_1, t)) \cos(\pi x_2), \\
    \dot{x}_2 &= \phantom{-}\pi A \cos(\pi f(x_1, t)) \sin(\pi x_2) \frac{\partial f}{\partial x}(x_1, t),
\end{align*}
with $ f(y, t) = \delta \sin(\omega t) \tsp y^2 + (1 - 2 \delta \sin(\omega t)) \tsp y $ and parameters $ A = 0.25 $, $\delta = 0.25 $, and $ \omega = 2 \pi $, see \cite{FP14} for more details. We choose the lag time $ \tau = 10 $ and define the test points $ x_i $ to be the midpoints of a regular $ 120 \times 60 $ box discretization of the domain $ [0, 2] \times [0, 1] $. To obtain the corresponding data points $ y_i = \Xi^\tau(x_i) $, where $ \Xi^\tau $ denotes the flow map, we use a Runge--Kutta integrator with variable step size. We then apply the singular value decomposition to the operator described above using a Gaussian kernel with bandwidth $ \sigma = 0.25 $. The resulting right singular functions are shown in Figure~\ref{fig:Double Gyre}.

\begin{figure}[tbh]
    \centering
    \begin{minipage}{0.32\textwidth}
        \centering
        \subfiguretitle{(a) $ \sigma_1 = 0.99 $}
        \includegraphics[width=\textwidth]{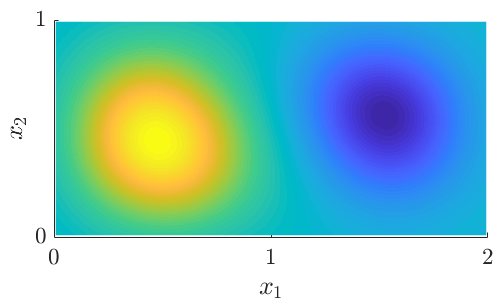} \\
    \end{minipage}
    \begin{minipage}{0.32\textwidth}
        \centering
        \subfiguretitle{(b) $ \sigma_2 = 0.98 $}
        \includegraphics[width=\textwidth]{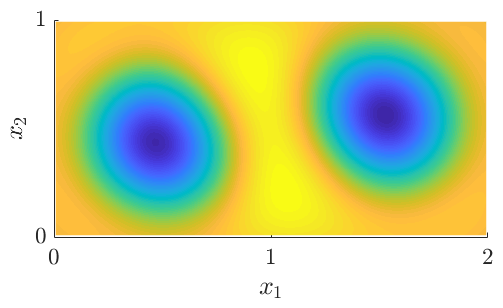} \\
    \end{minipage}
    \begin{minipage}{0.32\textwidth}
        \centering
        \subfiguretitle{(c) $ \sigma_3 = 0.94 $}
        \includegraphics[width=\textwidth]{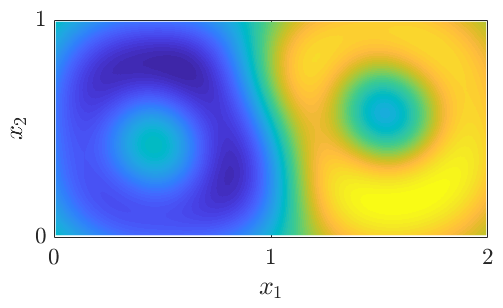} \\
    \end{minipage}
    \caption{Numerically computed singular values and right singular functions of $ \ecov[YY]^{-1/2} \ecov[YX] \tsp \ecov[XX]^{-1/2} $ associated with the double gyre flow.}
    \label{fig:Double Gyre}
\end{figure}

\end{example}

\section{Conclusion}
\label{sec:Conclusion}

We showed that the eigendecomposition and singular value decomposition of empirical RKHS operators can be obtained by solving associated matrix eigenvalue problems. To underline the practical importance and versatility of RKHS operators, we listed potential applications concerning kernel covariance operators, conditional mean embedding operators, and kernel transfer operators. While we provide the general mathematical theory for the spectral decomposition of RKHS operators, the interpretation of the resulting eigenfunctions or singular functions depends strongly on the problem setting. The eigenfunctions of kernel transfer operators, for instance, can be used to compute conformations of molecules, coherent patterns in fluid flows, slowly evolving structures in video data, or topic clusters in text data~\cite{KSM19}. Singular value decompositions of transfer operators might be advantageous for non-equilibrium dynamical systems. Furthermore, the decomposition of the aforementioned operators can be employed to compute low-rank approximations or their pseudoinverses, which might open up novel opportunities in statistics and machine learning. Future work includes analyzing connections to classical methods such as kernel PCA, regularizing finite-rank RKHS operators by truncating small singular values, solving RKHS operator regression problems with the aid of the pseudoinverse, and optimizing numerical schemes to compute the operator SVD by applying iterative schemes and symmetrization approaches.

\section*{Acknowledgements}

M.~M., S.~K., and C.~S were funded by Deutsche Forschungsgemeinschaft (DFG) through grant CRC 1114 (Scaling Cascades in Complex Systems, project ID: 235221301) and through Germany's Excellence Strategy (MATH\texttt{+}: The Berlin Mathematics Research Center, EXC-2046/1, project ID: 390685689). We would like to thank Ilja Klebanov for proofreading the manuscript and valuable suggestions for improvements.

\bibliographystyle{unsrt}
\bibliography{KTOSVD}

\begin{thebibliography}{10}

\bibitem{Reed80}
M.~Reed and B.~Simon.
\newblock {\em Methods of Mathematical Physics I: Functional Analysis}.
\newblock Academic Press Inc., 2nd edition, 1980.

\bibitem{aronszajn50reproducing}
N.~Aronszajn.
\newblock Theory of reproducing kernels.
\newblock {\em Transactions of the American Mathematical Society},
  68(3):337--404, 1950.

\bibitem{Schoe01}
B.~Sch{\"o}lkopf and A.~J. Smola.
\newblock {\em Learning with Kernels: Support Vector Machines, Regularization,
  Optimization and Beyond}.
\newblock MIT press, Cambridge, USA, 2001.

\bibitem{Berlinet04:RKHS}
A.~Berlinet and C.~Thomas-Agnan.
\newblock {\em Reproducing Kernel {H}ilbert Spaces in Probability and
  Statistics}.
\newblock Kluwer Academic Publishers, 2004.

\bibitem{StCh08}
I.~Steinwart and A.~Christmann.
\newblock {\em Support Vector Machines}.
\newblock Springer, 2008.

\bibitem{Smola07Hilbert}
A.~Smola, A.~Gretton, L.~Song, and B.~Sch\"{o}lkopf.
\newblock A {H}ilbert space embedding for distributions.
\newblock In {\em Proceedings of the 18th International Conference on
  Algorithmic Learning Theory}, pages 13--31. Springer-Verlag, 2007.

\bibitem{MFSS17}
K.~Muandet, K.~Fukumizu, B.~Sriperumbudur, and B.~Sch\"olkopf.
\newblock Kernel mean embedding of distributions: A review and beyond.
\newblock {\em Foundations and Trends in Machine Learning}, 10(1--2):1--141,
  2017.

\bibitem{SHSF09}
L.~Song, J.~Huang, A.~Smola, and K.~Fukumizu.
\newblock Hilbert space embeddings of conditional distributions with
  applications to dynamical systems.
\newblock In {\em Proceedings of the 26th Annual International Conference on
  Machine Learning}, pages 961--968, 2009.

\bibitem{Gruen12}
S.~Gr{\"u}new{\"a}lder, G.~Lever, L.~Baldassarre, S.~Patterson, A.~Gretton, and
  M.~Pontil.
\newblock Conditional mean embeddings as regressors.
\newblock In {\em International Conference on Machine Learing}, volume~5, 2012.

\bibitem{Klebanov2019rigorous}
I.~Klebanov, I.~Schuster, and T.~J. Sullivan.
\newblock A rigorous theory of conditional mean embeddings.
\newblock 2019.

\bibitem{Park2020MeasureTheoretic}
J.~Park and K.~Muandet.
\newblock A measure-theoretic approach to kernel conditional mean embeddings,
  2020.

\bibitem{Fukumizu13:KBR}
K.~Fukumizu, L.~Song, and A.~Gretton.
\newblock Kernel {B}ayes' rule: {B}ayesian inference with positive definite
  kernels.
\newblock {\em Journal of Machine Learning Research}, 14:3753--3783, 2013.

\bibitem{Fukumizu15:NBI}
K.~Fukumizu.
\newblock Nonparametric bayesian inference with kernel mean embedding.
\newblock In G.~Peters and T.~Matsui, editors, {\em Modern Methodology and
  Applications in Spatial-Temporal Modeling}. 2017.

\bibitem{KSM19}
S.~Klus, I.~Schuster, and K.~Muandet.
\newblock Eigendecompositions of transfer operators in reproducing kernel
  {H}ilbert spaces.
\newblock {\em Journal of Nonlinear Science}, 30:283--315, 2019.

\bibitem{KHM19}
S.~Klus, B.~E. Husic, M.~Mollenhauer, and F.~No\'e.
\newblock Kernel methods for detecting coherent structures in dynamical data.
\newblock {\em Chaos: An Interdisciplinary Journal of Nonlinear Science},
  29(12):123112, 2019.

\bibitem{KoWuNoSch18}
P.~Koltai, H.~Wu, F.~No{\'e}, and C.~Sch{\"u}tte.
\newblock Optimal data-driven estimation of generalized {M}arkov state models
  for non-equilibrium dynamics.
\newblock {\em Computation}, 6(1), 2018.

\bibitem{Weidmann}
J.~Weidmann.
\newblock {\em Lineare Operatoren in Hilbertr\"aumen}.
\newblock Teubner, 3rd edition, 1976.

\bibitem{Golub13}
G.H. Golub and C.F. {Van Loan}.
\newblock {\em Matrix Computations}.
\newblock John Hopkins University Press, 4th edition, 2013.

\bibitem{STJ04}
J.~Shawe-Taylor and N.~Christianini.
\newblock {\em Kernel Methods for Pattern Analysis}.
\newblock Cambridge University Press, 2004.

\bibitem{Kato80}
T.~Kato.
\newblock {\em Perturbation Theory for Linear Operators}.
\newblock Springer, Berlin, 1980.

\bibitem{Eubank15}
R.~Eubank and T.~Hsing.
\newblock {\em Theoretical Foundations of Functional Data Analysis with an
  Introduction to Linear Operators}.
\newblock Wiley, 1st edition, 2015.

\bibitem{EHN96}
H.~Engl, M.~Hanke, and A.~Neubauer.
\newblock {\em Regularization of Inverse Problems}.
\newblock Kluwer, 1996.

\bibitem{Baker1973}
C.~Baker.
\newblock Joint measures and cross-covariance operators.
\newblock {\em Transactions of the American Mathematical Society},
  186:273--289, 1973.

\bibitem{LeverEtAl:CCME17}
G.~Lever, J.~Shawe-Taylor, R.~Stafford, and C.~Szepesv{\'a}ri.
\newblock Compressed conditional mean embeddings for model-based reinforcement
  learning.
\newblock In {\em Association for the Advancement of Artificial Intelligence
  (AAAI)}, pages 1779--1787, 2016.

\bibitem{Stafford:ACCME18}
R.~Stafford and J.~Shawe-Taylor.
\newblock Accme: Actively compressed conditional mean embeddings for
  model-based reinforcement learning.
\newblock In {\em European Workshop on Reinforcement Learning 14}, 2018.

\bibitem{Gebhardt2019Robot}
G.H.W. Gebhardt, K.~Daun, M.~Schnaubelt, and G.~Neumann.
\newblock Learning robust policies for object manipulation with robot swarms.
\newblock In {\em IEEE International Conference on Robotics and Automation},
  2018.

\bibitem{SMKM20:CDO}
I.~Schuster, M.~Mollenhauer, S.~Klus, and K.~Muandet.
\newblock Kernel conditional density operators.
\newblock {\em The 23rd International Conference on Artificial Intelligence and
  Statistics (accepted for publication)}, 2020.

\bibitem{KNKWKSN18}
S.~Klus, F.~N\"uske, P.~Koltai, H.~Wu, I.~Kevrekidis, C.~Sch\"utte, and
  F.~No\'e.
\newblock Data-driven model reduction and transfer operator approximation.
\newblock {\em Journal of Nonlinear Science}, 2018.

\bibitem{MelzerEtAl:CCA}
T.~Melzer, M.~Reiter, and H.~Bischof.
\newblock Nonlinear feature extraction using generalized canonical correlation
  analysis.
\newblock In Georg Dorffner, Horst Bischof, and Kurt Hornik, editors, {\em
  Artificial Neural Networks --- ICANN 2001}, pages 353--360, Berlin,
  Heidelberg, 2001. Springer Berlin Heidelberg.

\bibitem{FP14}
G.~Froyland and K.~Padberg-Gehle.
\newblock Almost-invariant and finite-time coherent sets: Directionality,
  duration, and diffusion.
\newblock In W.~Bahsoun, C.~Bose, and G.~Froyland, editors, {\em Ergodic
  Theory, Open Dynamics, and Coherent Structures}, pages 171--216. Springer New
  York, 2014.

\end{thebibliography}

\appendix

\section{Appendix}

\subsection{Proof of block SVD}
\label{ap:block_operator_svd}

\begin{proof}[Lemma \ref{lem:block_operator_svd}]
  Let $A$ admit the SVD given in \eqref{eq:block_svd}. Then by the definition of $T$, we have
  \begin{equation*}
    T (\pm u_i , v_i) = (A v_i , A^* u_i) = \pm \sigma_i (\pm u_i, v_i)
  \end{equation*} for all $i \in I$.
  For any element $(f,h) \in \mspan\{(\pm u_i , v_i)\}_{i \in I}^{\perp}$, we can immediately deduce
  \begin{equation*}
    0 = \innerprod{(f,h)}{(\pm u_i,v_i)}_{\exsum} = \pm \innerprod{f}{u_i}_F + \innerprod{h}{v_i}_H
  \end{equation*}
  for all $i \in I$ and hence $f \in \mspan\{u_i \}_{i \in I}^\perp$ and  $h \in \mspan\{v_i \}_{i \in I}^\perp$. Using the SVD of $A$ in \eqref{eq:block_svd}, we therefore have
  \begin{equation*}
    \restr{T}{\mspan\{(\pm u_i , v_i)\}_{i \in I}^{\perp}} = 0.
  \end{equation*}
  It now remains to show that
  $\left\{ \tfrac{1}{\sqrt{2}} (\pm u_i,v_i) \right\}_{i \in I}$ is an orthonormal system in $F \exsum H$, which is clear since $\innerprod{(\pm u_i,v_i)}{( \pm u_j,v_j)}_{\exsum} = 2\,\delta_{ij}$ and $\innerprod{(-u_i,v_i)}{(u_j,v_j)}_{\exsum} = 0$ for all $i,j \in I$. Concluding, $T$ has the form~\eqref{eq:block_eigendecomposition} as claimed.
\end{proof}

\subsection{Derivation of the empirical CCA operator}
\label{app:empirical_cca}

The claim follows directly when we can show the identity
\begin{equation*}
    \Phi^\top (\Phi \Phi^\top)^{-1/2} = \gram[\Phi]^{-1/2} \Phi^\top
\end{equation*} and its analogue for the feature map $\Psi$.
Let $\gram[\Phi] = U \Lambda U^\top$ be the eigendecomposition
of the Gramian. We know that in this case we have the SVD of the operator
$\Phi \Phi^\top = \sum_{i \in I} \lambda_i (\lambda_i^{-1/2}\Phi u_i) \otimes (\lambda_i^{-1/2} \Phi u_i)$, since
\begin{equation*}
 \innerprod{\lambda_i^{-1/2}\Phi u_i}{\lambda_j^{-1/2}\Phi u_j}_\mathscr{H}
 = \lambda_i^{-1/2} u_i \gram[\Phi] u_j \lambda_j^{-1/2} = \delta_{ij}.
\end{equation*}
We will write this operator SVD for simplicity as $\Phi \Phi^\top = (\Phi U \Lambda^{-1/2}) \Lambda
(\Lambda^{-1/2} U \Phi^\top)$ with an abuse of notation.
Note that we can express the inverted operator square root elegantly in this form as $(\Phi \Phi^\top)^{-1/2} = (\Phi U \Lambda^{-1/2}) \Lambda^{-1/2}
(\Lambda^{-1/2} U \Phi^\top) = (\Phi U)  \Lambda^{-3/2} (U \Phi^\top) $.
Therefore, we immediately get
\begin{align*}
 \Phi^\top (\Phi \Phi^\top)^{-1/2}
 &= \Phi^\top (\Phi U \Lambda^{-3/2} U^\top \Phi^\top)\\
 &= \gram[\Phi] U \Lambda^{-3/2} U^\top \Phi^\top \\
 &= U \Lambda U^\top U \Lambda^{-3/2} U^\top \Phi^\top \\
 &= U \Lambda^{-1/2} U^\top \Phi^\top = \gram[\Phi]^{-1/2} \Phi^\top,
\end{align*} which proves the claim.
In the regularized case, all operations work the same with an additional
$\epsilon$-shift of the eigenvalues, i.e., the matrix $\Lambda$ is
replaced with the regularized version $\Lambda + \epsilon\mathrm{I}$.

\end{document}